\documentclass[12pt]{amsart}

\usepackage{amsmath,amsthm,amssymb,mathrsfs,amsfonts,verbatim,enumitem}
\usepackage{color}

\usepackage{etoolbox} 
\usepackage{esvect}
\usepackage{bbm}
\usepackage[all,tips]{xy}
\usepackage{graphicx,ifpdf}
\ifpdf
\fi

\newtheorem{thm}{Theorem}[section]
\newtheorem{lem}[thm]{Lemma}
\newtheorem{cor}[thm]{Corollary}
\newtheorem{prop}[thm]{Proposition}

\theoremstyle{definition}

\theoremstyle{remark}

\numberwithin{equation}{section}

\newcommand{\rmnum}[1]{\romannumeral #1}
\newcommand{\Rmnum}[1]{\expandafter\@slowromancap\romannumeral #1@}
\newcommand{\bv}{\mathbf{v}}
\newcommand{\bw}{\mathbf{w}}
\newcommand{\vre}{\varepsilon}
\newcommand{\bp}{\mathbf{p}}
\newcommand{\br}{\mathbf{r}}
\newcommand{\DI}{\mathrm{DI}}

\newcommand{\R}{\mathbb{R}}
\newcommand{\N}{\mathbb{N}}
\newcommand{\SL}{\operatorname{SL}}
\newcommand{\GL}{\operatorname{GL}}
\newcommand{\Z}{\mathbb{Z}}
\newcommand{\Q}{\mathbb{Q}}
\newcommand{\C}{\mathbb{C}}
\newcommand{\diag}{\mathrm{diag}}
\newcommand{\spa}{\mathrm{span}}
\newcommand{\EE}{\mathbf{e}}
\newcommand{\sm}{\smallsetminus}
\newcommand{\Ad}{\operatorname{Ad}}
\newcommand{\til}{\widetilde}

\begin{document}

\title{Invariant measures for solvable groups and Diophantine approximation}

% Information for first author
\author{Ronggang Shi}
% Address of record for the research reported here
\address{School of Mathematical Sciences, Tel Aviv University, Tel Aviv 69978, Israel, and
School of Mathematical Sciences, Xiamen University, Xiamen 361005, PR China}
% Current address
%\curraddr{Department of Mathematics, Ohio State
 %University, Columbus, Ohio 43210}
 \email{ronggang@xmu.edu.cn}
% \thanks will become a 1st page footnote.
\thanks{The authors are supported by NSFC (11201388), NSFC (11271278),
  BSF grant 2010428, and 
 ERC starter grant DLGAPS 279893.}

\author{Barak Weiss}
\address{School of Mathematical Sciences, Tel Aviv University, Tel
  Aviv 69978, Israel} 
% Current address
%\curraddr{Department of Mathematics, Ohio State
 %University, Columbus, Ohio 43210}
 \email{barakw@post.tau.ac.il}

% General info

\date{}

%\dedicatory{This paper is dedicated to my advisor.}


\begin{abstract}
We show that if $\mathcal{L}$ is a line in the plane containing a
badly approximable vector, then almost every point in $\mathcal{L}$
does not admit an improvement in Dirichlet's theorem. Our proof relies
on a measure classification result for certain measures invariant
under a non-abelian two dimensional group on the
homogeneous space $\SL_3(\R)/\SL_3(\Z)$. Using the measure
classification theorem, we
reprove a result of Shah about planar nondegenerate curves (which are
not necessarily analytic), and prove analogous results for the
framework of Diophantine approximation with weights. We also show
that there are line segments in $\R^3$, which do contain badly
approximable points, and for which all points do admit an improvement
in Dirichlet's theorem. 
\end{abstract}

\maketitle

\markright{}

\section{Introduction}\label{sec;intro}
A classical result in Diophantine approximation is Dirichlet's theorem
which asserts that for any $\bv \in \R^n$ and any $Q \geq 1$ there are $q
\in \N$ and $\bp \in \Z^n$ such that 
$$
\|q \bv - \bp\| < \frac{1}{Q^{1/n}} \ \text{ and } \ q \leq Q.
$$
The norm used here and throughout this paper is the sup-norm on
$\R^n$. 
Let $\sigma \in (0,1)$. Following Davenport and Schmidt \cite{ds}, we
say that $\bv$  {\em admits a
$\sigma$-improvement for Dirichlet's theorem}, and write $\bv \in \DI(\sigma)$, if 
for all sufficiently large $Q$, there are $q \in \N$ and $\bp \in
\Z^n$ such that 
$$
\|q \bv - \bp\| < \frac{\sigma}{Q^{1/n}} \ \text{ and } \ q < \sigma Q.
$$
Finally we say that $\bv$ {\em admits no improvement in Dirichlet's
  theorem} if $\bv \notin \bigcup_{\sigma <1} \DI(\sigma)$. It is known that almost
every $\bv \in \R^n$ (with respect to Lebesgue measure) admits no
improvement in Dirichlet's theorem. It is an interesting problem to
decide, given a measure $\mu$ on $\R^n$, whether $\mu$-a.e. $\bv$
admits no improvement in Dirichlet's theorem. See \cite{ds, KW_di}
for some results and questions in this direction. 

In a recent breakthrough, Shah \cite{s09} showed that if $\mu$ is the
length measure on an analytic curve in $\R^n$, which is not contained
in any affine hyperplane, then $\mu$-a.e. $\bv$ admits no
improvement in Dirichlet's theorem. 
For certain 
fractal measures $\mu$ in $\mathbb R^2$,
the same conclusion is obtained in \cite{s120} and \cite{s12}.
These works leave open the
question of measures which are length measures on lines. In this
direction, Kleinbock \cite{almost vs no} showed that for any line
$\mathcal{L} \subset \R^n$ which is not contained in $\DI(\sigma_0)$
for some $\sigma_0 >0$, for almost every $\bv
\in \mathcal{L}$ (w.r.t. length measure on $\mathcal{L}$), there is
  $\sigma = \sigma(\bv)$ such that $\bv \notin \DI(\sigma)$. Our first
  result strengthens this conclusion under a stronger hypothesis, for
  planar lines. Recall that $\bv$ is called {\em badly approximable}
  if there is $c>0$ such that for any $q \in \N$ and $\bp \in \Z^n$,
  $\|q\bv - \bp\| \geq \frac{c}{q^{1/n}}.$ 

\begin{thm}\label{thm;dirichlet}
 Suppose that a line  $\mathcal L$ in $\mathbb R^2$ contains a badly approximable vector. 
Then almost every element of $\mathcal L$ (w.r.t. length measure)
admits no improvement in Dirichlet's theorem. 
\end{thm}

%\begin{proof}
%It follows from \cite[Theorem 6.1]{abv} that under the condition of 
%(\ref{eq;condition})
%the line  $\mathcal L$ contains a badly approximable  vector. 
%Therefore the conclusion follows from Theorem \ref{thm;dirichlet}.
%\end{proof}

Another question raised by Shah's work is to what extent one can relax
the hypothesis of the analyticity of the curve. A map $\varphi : [0,1]
\to \R^n$ is called {\em nondegenerate} if it is $n$ times
continuously differentiable, and for almost every $s$, the Wronskian
determinant of $\varphi'(s)$ does not vanish (i.e. the vectors  $\varphi'(s), \varphi''(s), \cdots ,
\varphi^{(n)}(s)$ are linearly independent in $\R^n$). It is clear that analytic
curves not contained in affine hyperplanes are nondegenerate, and one
may expect that the conclusion of Shah's theorem holds under
this weaker hypothesis. This was proved by Shah in the case $n=2$ by
adapting the 
method of \cite{s09}. We obtain a simpler proof. That is we show: 
\begin{thm}\label{prop;curve}
Let $\varphi : [0,1]\to \mathbb R^2$ be a
nondegenerate curve. 
Then for almost every $s\in [0,1]$ (with respect to Lebesgue measure), $\varphi(s)$ admits no
improvement in Dirichlet's theorem. 
\end{thm}

A similar proof of Theorem \ref{prop;curve} was obtained independently
by Manfred Einsiedler. 

Our proofs rely on results in homogeneous dynamics. Before stating
them we introduce some notation, to be used in \S\ref{sec;intro}--\S\ref{sec: curve}. 
Let $G := \SL_3(\R), \, \Gamma := \SL_3(\Z), \, X := G/\Gamma$, so that
$X$ is the space of unimodular lattices in $\R^3$. This is a space on
which any subgroup of $G$ acts by left-translations preserving the
$G$-invariant Borel 
probability measure $m$ induced by Haar measure on $G$. 
 For $\bv=(v_1, v_2)^{\mathrm{tr}}\in \mathbb R^2$, $t\in \mathbb R$
 and $\br = (r_1,
r_2)\in \mathbb R^2_{>0}$  with $r_1+r_2=1$,  we set 
\begin{equation}\label{eq;notation}
f_{t}^{(\br)}:=
\left(
\begin{array}{ccc}
e^{r_1t} & 0 & 0 \\
0 & e^{r_2t} & 0 \\
0 & 0 & e^{-t}
\end{array}
\right),
\quad
u(v_1, v_2):=u(\mathbf v):=
\left(
\begin{array}{ccc}
1 & 0 & v_1 \\
0 & 1 & v_2 \\
0 & 0 & 1
\end{array}
\right),
\end{equation}
and let $\bar{u} = \pi \circ u, $ where $\pi: G \to G/\Gamma$ is the
natural quotient map. Theorem \ref{thm;dirichlet} follows from: 
\begin{thm}\label{thm;equiline}
Let $x_0\in X$,  $a, b \in \R$ and let $I , J\subset \R$ be bounded intervals, and
suppose there is a compact $K \subset X$ such that 
\begin{equation}\label{eq: weird condition}
\text{for all } t
\geq 0 \text{ there is } s_t \in J \text{ with } f_t^{(\br)} {u}(s_t, as_t+b)x_0
\in K.
\end{equation}
Let $\nu$ be a probability measure on $I$ which is absolutely
continuous with respect to Lebesgue measure. 
  Then 
%for any $x_0 \in X$ 
%probability measure $\nu$ on $\R$ which is absolutely
 % continuous with respect to Lebesgue measure, any $x_0 \in X$, 
%  interval 
%$I$ with
%positive length   
%and 
for any $\psi\in C_c(X)$ one has
\[
\frac{1}{T}\int_0^{T}\int_I  \psi(f_t^{(\br)} {u}(s, as+b)x_0)\, d\nu(s)\, dt\to_{T
  \to \infty}
\int _X \psi \, dm; 
\]
that is, $\frac{1}{T} \int_0^T \left(f_{t}^{(\br)} \right)_* \bar{\nu} \, dt \to_{T \to \infty} m$ in the weak-*
topology on Borel probability measures on $X$, where $\bar{\nu}$ is
the image of $\nu$ under the map $s  \mapsto {u}(s, as+b)x_0. $
\end{thm}

Similarly, Theorem \ref{prop;curve} follows from: 
\begin{thm}\label{prop;equicurve}
Let $\varphi: [0,1]\to \mathbb R^2$ be a nondegenerate curve. 
 Then for  any
$\psi\in C_c(X)$ and any probability measure $\nu$ on $[0,1]$ which is
absolutely continuous with respect to Lebesgue measure, one has
\[
\frac{1}{T}\int_0^{T}\int_0^1 \psi \left(f_t^{(\br)}\bar{u}(\varphi(s)) \right) \,  d\nu(s)\, dt\to_{T
  \to \infty}\int _X \psi\, dm. 
\]
\end{thm}

\begin{comment}
\begin{thm}\label{thm;pointwise}
Let $a\in \mathbb R\setminus \mathbb Q$,   $(b_1, b_2)\in \mathbb R^2$
and $\mathbf a= (a_1, a_2)\in \mathbb R^2_{>0}$ with $a_1+a_2=1$. 
Then 
for almost every 
$s\in \mathbb R$ and any $\psi\in C_c(X)$ we have
\begin{equation}\label{eq;equidis}
\frac{1}{T}\int_0^T \psi (f_{t\mathbf a} u(s+ b_1, as +b_2 )x_0)\, dt=\int_X \psi (x)\, dx.
\end{equation}
\end{thm}

\end{comment}

Theorems \ref{thm;equiline} and \ref{prop;equicurve} in turn follow
from the following measure classification result: 

\begin{thm}\label{thm;measure}
Let $U$ (resp.~$F$) be 
a one parameter unipotent (resp.~diagonalizable) 
   subgroup of $G$. 
Suppose that 
$U$ is normalized by $F$, $FU$ is nonabelian and
$F$ does not fix any nonzero vector of $\mathbb R^3$. 
Then the action of $FU$ on $X$ is uniquely ergodic, i.e. $m$ is the only $FU$-invariant probability measure on $X$. 
\end{thm}

Our method of proof allows a generalization to `Diophantine
approximation with weights', which we now describe. 
%Let $\br := (r_1,
%r_2)$ where $r_1, r_2$ are positive real numbers with $r_1+r_2=1.$
Let $\br = (r_1, r_2)^{\mathrm{tr}}$ be as above. Following \cite{kleinbock duke} we
say that $\bv \in \R^2$ is {\em badly approximable w.r.t. weights
  $\br$} if there is $c>0$ such that for all $q \in \N$, all $\bp
\in \Z^2$, and $i=1,2$  we have
$$
|qv_i - p_i|^{1/r_i} \geq \frac{c}{q}. 
$$
Also, following \cite{KW_di} we say that {\em $\bv$ admits no improvement in
  Dirichlet's theorem w.r.t. weights $\br$} if there does not exist
$\sigma \in (0,1)$ such that for all sufficiently large $Q$, there is
a solution $q \in \N,\, \bp \in \Z^2$ to the inequalities
$$
|qv_i-p_i| < \frac{\sigma}{Q^{r_i}},\ i=1,2,\ q < \sigma Q. 
$$

We show: 
\begin{thm}\label{thm: weights}
For any $\br$ as above, the following hold: \begin{itemize}
\item[(\rmnum{1})]
Suppose $\mathcal{L}$ is a line in $\R^2$ which
contains one point which is badly approximable w.r.t. weights
$\br$. Then almost every $\bv \in \mathcal{L}$ (w.r.t. the length
measure on $\mathcal{L}$) admits no improvement in Dirichlet's
theorem w.r.t. weights $\br$. 
\item[(\rmnum{2})]
Let $\varphi: [0,1] \to \R^2$ be a
nondegenerate curve. Then for almost every $s \in [0,1]$
(w.r.t. Lebesgue measure), $\varphi(s)$ admits no improvement in
Dirichlet's theorem w.r.t. weights $\br$. 
\end{itemize} 
\end{thm}
Theorem \ref{thm: weights}(\rmnum{2}) was proved for nondegenerate
analytic curves in $\R^n$, in \cite{s10}. 
The hypothesis of Theorem \ref{thm;dirichlet}
and \ref{thm: weights}(\rmnum{1}) can be verified in many
cases. In light of recent work of Badziahin-Velani \cite{bv} and An-Beresnevich-Velani \cite{abv},
we obtain: 
\begin{cor}
Suppose that $\mathcal L$ is a line in $\mathbb R^2$ given by the equation 
$
\mathrm{y}=a\mathrm{x}+b
$
where $a\neq 0$.
If
\begin{equation}\label{eq;condition}
\liminf_{q\to \infty}|q|^{\frac{1}{r}-\vre} \min_{\bp \in \Z^2} \| q(a,b) - \bp
\| >0
\quad \mbox{where } r=\min\{r_1, r_2\}
\end{equation}
for some $\vre >0$, then almost every $\bv \in \mathcal{L}$ admits no improvement in 
Dirichlet's theorem  w.r.t. weights $\br$. Moreover the same conclusion holds
 if $a\in \mathbb Q$ and (\ref{eq;condition}) holds for $\vre=0$.
\end{cor}

In \S \ref{sec: examples} we give several examples showing the
necessity of the hypotheses in our
theorems. In particular we show in Theorem \ref{thm: example for dirichlet}, that the analog of
Theorem \ref{thm;dirichlet} fails in dimension $n=3.$ 

\medskip

{\bf Acknowledgements.} We are grateful to Jinpeng An,  Manfred Einsiedler, Dmitry
Kleinbock and Elon Lindenstrauss for helpful discussions.  
%{\color{blue} Feel free to add any more acknowledgements.} 

 \section{Invariant measure for solvable groups}
In this section we prove Theorem \ref{thm;measure}. As we will show in
\S \ref{sec: examples}, it is not possible to relax the
hypotheses of the theorem.

Let the notation be as in the statement of Theorem \ref{thm;measure},
and let $F=\{f_t: t \in \R\}$ where $t\mapsto f_t$ is a group homomorphism from
$\mathbb R\to F$. 
% so that $U$ is contained in the expanding horospherical 
%subgroup of $g_1$, i.e.~
%\[
%U\subset U^+:=\{h\in G: f_{-t}hf_t\to_{ t\to \infty} e\}
%\]
%where $e$ is the identity in $G$. 
%Note that we do not restrict ourselves at this point to $f_t$ as in
%\eqref{eq;notation}. 
Let $\mu$ be an $FU$-invariant Borel probability
measure on $X$. Our goal is to show that $\mu=m$, and we can assume
with no loss of generality that $\mu$ is ergodic for the action of
$FU$. 

We can decompose $\mu$ into its $U$-ergodic components. That is we
write $\mu=\int _X m_x\, d\mu(x)$ where each $m_x$ is $U$-invariant
and ergodic. According to Ratner's measure classification theorem
\cite{ratner}, for every $x$ there is a closed connected subgroup
$H=H_x$ such that $\overline{Ux} = Hx$ and $m_x$ is the unique
$H$-invariant measure on $Hx$ induced by the Haar measure on
$H$. Also, since $\mu$ is $F$-invariant, by the Poincar\'e recurrence
theorem, for almost every $x$ and $m_x$-a.e. $y$, the orbit $Fy$ is
recurrent in both positive and negative times, i.e. there are $t_n \to +\infty$ and $t'_n \to -\infty$ such
that 
\begin{equation}\label{eq: Poincare} 
f_{t_n}y \to y \text{ and } f_{t'_n} y\to y.
\end{equation}  

We will need the following result:
%Let $\mathcal P(X)$ be the space of probability measures on $X$ with the weak$^*$
%topology.  
%The group $G$ acts continuously on $\mathcal P(X)$ via the pushforward map
%$\nu\to g\nu$. 
%For every $x\in X$ let $m_x$ be the homogeneous measure associated to  $\overline{Ux}$, i.e.~there
%exists a (unique) connected subgroup $H_x$ of $G$ such that  $\overline{Ux}=H_x x$ and $m_x$ is the unique
%$H_x$ invariant probability measure supported on $H_x x$.
%We claim that the map $\kappa:X\to \mathcal P(X)$ where $\varphi(x)=m_x$ is measurable. 
%({\color{red} find a reference or prove it using linearization method})
%Moreover we have the following ergodic decomposition: 
%\[
%\mu=\int _X m_x\, d\mu(x).
%\]
%Since $F$ normalizes $U$ we have 
%$f_tm_x=m_{f_tx}$ for every $x\in X$. In other words the following diagram commutes:
%\[
%  \xymatrix{
%        X \ar[r]^\kappa \ar[d]_{f_t} & \mathcal P(X) \ar[d]^{f_t} \\
%        X \ar[r]_{\kappa}       & \mathcal P(X) }
%\]
%Therefore the probability measure $\kappa\mu$ on $\mathcal P(X)$ is $F$ invariant and ergodic.  

\begin{thm}[Mozes \cite{m95}, see also \cite{mt96}]\label{thm;mozes}
There exists %$X'\subset X$ with $\mu(X')=1$
%and 
a closed subgroup $H$ of $G$ generated by one-parameter unipotent
subgroups and containing $U$ 
such that the following hold:
\begin{enumerate}[label=(\roman*)]
\item For $\mu$-almost every $x\in X$ we have $H_x=H$.
\item The group $H$ is normalized by $F$ and conjugation by $F$ preserves the Haar measure of $H$.
%\item The measure $m_x$ is generic, i.e.~$\frac{1}{T}\int_0^T f_t m_x\, dt\to \mu $
%as $T\to \pm \infty$.
\end{enumerate}
 \end{thm}

Let $\{h_t: t\in \mathbb R\}$ be a $1$-parameter  subgroup of $G$. 
We say that $\{h_tx : t \geq 0\}$
(respectively $\{h_tx: t \leq 0\}$) is {\em divergent} if for any
compact $K \subset X$ there is $t_0$ such that for all $t>t_0$ (resp.,
all $t<t_0$), $h_tx\notin K$. We will need the following well-known
fact: 
\begin{prop}\label{prop: bounded denominators}
If $\rho: G \to \GL(V)$ is a representation defined over $\Q$, and $v
\in V(\Q)$ such that $\rho(h_tg)v \to_{t \to + \infty} 0$, then $\{h_t
\pi(g): t\geq 0\}$ is divergent. The analogous statement replacing
$+\infty$ with $-\infty$ and $t\geq0$ with $t\leq 0$ also holds. 
\end{prop}

\begin{proof}
This follows from a standard bounded denominators argument, see
e.g. \cite[Prop. 3.1]{Cassels}. 
\end{proof}

We let $E_{ij}$ be the matrix whose matrix coefficient in the $i$th
row and $j$th column is 1, and 0 elsewhere. 
Set
\begin{equation}\label{eq;u}
U_{ij}:=\{\exp (sE_{ij}) : s\in \mathbb R\}.
\end{equation}
Let $U^+ := \langle U_{12}, U_{13}, U_{23}\rangle$ be the upper
triangular unipotent group. 
We will need the following:

\begin{prop}\label{prop: upper triangular}
Let $x \in X$ such that $U^+x$ is closed. Then for any  1-parameter
subgroup $\{h_t\}$ of the diagonal group, at least one
of the two trajectories $\{h_tx: t \geq 0\}, \, \{h_tx: t \leq 0\}$ is
divergent. 
\end{prop}

\begin{proof}
First suppose that $x$ is the point corresponding to the identity
coset $\Gamma$, that is $x = \pi(e)$ where $e$ is the identity element
of $G$. There is a natural action of $G$ on $\R^3$ by linear
transformations and a corresponding induced action on the second
exterior power $\bigwedge^2 \R^3$. Let $\EE_1, \EE_2, \EE_3$ be the
standard basis of $\R^3$ and let $\bv_{12}
:= \EE_1 \wedge \EE_2\in \bigwedge^2 \R^3$. The vectors $\EE_1,
\bv_{12}$ 
%are $U^+$-invariant and 
are eigenvectors for the diagonal group, and we
let $\chi_1, \chi_2$ be the corresponding characters. That is, if $a =
\diag(e^s, e^t, e^{-(s+t)})$, then: 
$$a \EE_1 =\chi_1(a) \EE_1, \ \text{ where } \chi_1(a) = e^{s}$$ and 
$$a \bv_{12} = \chi_2(a) \bv_{12},\ \text{ where } \chi_2(a) =e^{s+t}.$$
For any one-parameter
diagonal subgroup $\{h_t\}$, at least one of the two restrictions
$\chi_i|_{h_t}, \ i=1,2$ is not trivial. This implies that $h_t \EE_1 \to 0$ or
$h_t \bv_{12} \to 0$ as $t$ tends to either $+\infty$ or
$-\infty$, and we apply Proposition \ref{prop: bounded denominators}. 

Now suppose that $x = \pi(g)$ for some 
$g \in G$. For definiteness, assume that $h_t \EE_1 \to_{t \to
  +\infty} 0$ (if not, replace $\EE_1$ by $\bv_{12}$ or $+\infty$ by $-\infty$). Since closed orbits for unipotent groups are of
finite volume, $g^{-1}U^+g \cap \Gamma$ is a lattice in $U^+$. Therefore the group 
$g^{-1}U^+g$ is defined over $\mathbb Q$. 
So both the normalizers of $U^+$ and $g^{-1}U^+g$ are $\mathbb Q$-parabolic subgroups 
of $G$, and hence are conjugate over $\Q$. This implies that there exists $g_0\in \SL_3(\mathbb Q)$ such that 
\[
g^{-1}U^+g=g_0^{-1} U^+ g_0.
\]
It follows that $ng_0=g$ where $n\in N_G(U^+)$. Note that both $\EE_1$
and $\bv_{12}$ are eigenvectors for the upper triangular group
$N_G(U^+)$, so we write $n\EE_1 = c\EE_1$ for some $c \in \R$. Therefore we have 
$$h_t g g_0^{-1}{\mathbf e_1}= h_t n \EE_1 = ch_t
\EE_1 \to 0.$$
Since $g_0 \in \SL_3(\Q),$ $g_0^{-1} \EE_1$ is a $\Q$-vector. 
 Applying again Proposition \ref{prop: bounded denominators} (with $g_0^{-1}
\EE_1$ instead of $\EE_1$) we see that the trajectory $\{h_t x\}$ is
divergent. 
\end{proof}

Let $H_0 \cong \SL_2(\R)$ denote the subgroup of $G$ generated by
$U_{12}$ and $U_{21}$. We will need a similar fact for $H_0$. 
\begin{prop}\label{prop: copy of sl2}
Let $x \in X$ such that $H_0x$ is closed, and let $\{h_t\}$ be a
one-parameter subgroup of the group of diagonal matrices which is not
contained in $H_0$. Then $\{h_tx: t\geq 0\}$ and $\{h_tx: t \leq 0\}$
are both divergent. 
\end{prop}

\begin{proof}
First suppose that $x = \pi(e)$ and consider
the vector $ \bv_{12} = \EE_1 \wedge \EE_2 \in \bigwedge^2 \R^3$ of
the previous proof, along with the vector $\EE_3$. For any 1-parameter
group $\{h_t\}$ not contained in $H_0$, possibly after switching the
roles of $+\infty$ and $-\infty$, we have $h_t \EE_3 \to_{t \to
  +\infty} 0$ and $h_t \bv_{12} \to_{t \to -\infty} 0.$ Therefore the
claim follows from Proposition \ref{prop: bounded denominators}. 

Now assume that $x = \pi(g)$ for some $g \in G$. The group $H_0$ is
the stabilizer of the vector $\bw:= \bv_{12} \oplus \EE_3$ in the
representation $W:=\bigwedge^2 \R^3 \oplus \R^3$. Moreover $\bw$ 
represents the unique splitting of $\R^3$ into a direct sum
decomposition of a 2-dimensional and 1-dimensional space which is
left invariant by $H_0$. Consider the group $H' := g^{-1} H_0 g$ and
the vector $\bw':=g^{-1} \bw \in W$. Then $\bw'$ represents the unique
splitting into a direct sum decomposition as above, which is $H'$
invariant. Also, since $Hx$ is closed, it is of finite volume and $H'
\cap \Gamma$ is a lattice in $H'$. This implies that $H'$ is defined
over $\Q$. 

Now let $\iota: \C \to \C$ be any field automorphism. The
map $\iota$ acts on $G$ (by its action on matrix entries) and on $W$
(by its action on vector coefficients) in a compatible way, and
$\iota(H')=H'$ since $H'$ is defined over $\Q$. This implies that 
$\iota(\bw')$ also represents the unique splitting
$\iota(H')$-invariant decomposition of $W$ into a 1- and 2-dimensional
subspace. Since the dimensions of these two subspaces are different,
$\iota$ also preserves each subspace in this splitting, that is,
$\iota$ preserves $\bv_{12}':= g^{-1} \bv_{12}$ and $\EE':= g^{-1} \EE_3$. Since this is
true for any field automorphism $\iota$, 
$\bv'_{12}$ and $\EE'$ are
$\Q$-vectors in $\R^3$ and $\bigwedge^2 \R^3$ respectively, and 
$$h_t
g \EE' = h_t \EE_3 \to_{t \to +\infty} 0, \ h_t g \bv'_{12} = h_t \bv_{12}
\to_{t \to -\infty} 0.$$
Thus the claim follows using Proposition \ref{prop: bounded
  denominators} with
$\bv'_{12}$ and $\EE'$. 
\end{proof}

\begin{proof}[Proof of Theorem \ref{thm;measure}] Let $F$ and $U$ be
  as in the statement of the theorem, and for an $FU$-invariant
  ergodic measure $\mu$, let $H$ be as in Theorem \ref{thm;mozes}. 
   We will prove Theorem \ref{thm;measure} by showing $H=G$, and to
   this end we will assume by contradiction that $H \neq G$,  consider
   various possibilities for the triple 
$(F, U, H)$, and derive a contradiction in each case. 
%  We note that for any $g\in G$ 
%  there is a conjugation of dynamical systems
%\[
%(X, FU) \xrightarrow{g} (X, gFU g^{-1}).
%\]
%Therefore it suffices to consider a representative  of the triple $(F, U, H)$ in 
%every conjugacy class.
%Let $\tau: G\to G$ be the transpose  inverse automorphism, i.e.~$\psi(g)=(g^{-1})^{\mathrm {tr}}$. 
%Then $\tau$ induces an homeomorphism of the space $X$
%which  we continue to denote by $\tau$. The measure $\tau \mu$ is $(FU)^{\mathrm {tr}}$ invariant 
%and ergodic. Therefore the validity of Theorem \ref{thm;measure} for the pair $(F, U)$  implies that 
%for the pair $(F^{\mathrm {tr}}, U^{\mathrm {tr}})$.

%The following result   is  a special case of  the measure rigidity theorem  of Einsiedler, Katok and Lindenstrauss
% \cite[Theorem 1.3]{ekl06}.
% \begin{thm}[\cite{ekl06}]\label{thm;ekl}
% Let $\mu$ be a probability measure on $X$.
% Let $A$ be a connected two dimensional diagonalizable subgroup of $G$. 
%  Suppose that $\mu$ is invariant under $A$ and some nontrivial $A$ normalized  unipotent  subgroup of $G$, then $\mu=m$. 
% \end{thm}

Let $\mathfrak{h}, \mathfrak{u}$ denote respectively the Lie algebras
of $H$ and $U$. 
The key observation is the following. 
Since conjugation by $f_1$ preserves the volume of $H$ and
$\mathfrak{u} \subset \mathfrak{h}$, and since conjugation by $f_1$ does not preserve the
volume of $U$, the adjoint action of $f_1$ on $\mathfrak{u}$ is
nontrivial and hence $\mathfrak{h}$ must contain eigenvectors of
$\mathrm{Ad}(f_1)$ with both positive and negative eigenvalues.

The group of automorphisms of $G$ is generated by inner automorphisms
(conjugation) and the automorphism $g \mapsto (g^{-1})^{\mathrm{tr}}.$
With no loss of generality we can apply an automorphism of $G$ and a
reparametrization of $F$ to the
triple $(F, U, H)$ to assume:
\begin{enumerate}
\item
$f_t = \diag(e^{t}, e^{at}, e^{bt})$ where $1 \geq a >0> b, a+b = -1$
(since such one-parameter subgroups fill up a fundamental domain for
the action of the automorphism group of $G$ on the diagonal group, and
since $f_t$ does not preserve a vector in $\R^3$). 
\item
$U$ is contained in the upper triangular group $U^+ $
(since, by reparameterizing $f_t$, we may assume it acts on $U$
by expansion). 
\item
The subgroup $H \cap U^-,$ where $U^-$ is the lower triangular
unipotent subgroup $\langle
U_{21}, U_{31}, U_{32} \rangle,$ contains a nontrivial group $N$
(whose Lie algebra is denoted by $\mathfrak n$) such
that $F$ normalizes $N$ and acts
on its Lie algebra by a strict contraction (since the action of $F$ on $H$ preserves Haar measure on $H$
so there must be a subgroup which is contracted).  
\end{enumerate} 

Suppose first that $a=1$, so that $b=-2$. 
%That is $\{f_t\}$ is the
%group given by \eqref{eq;notation}. 
In this case
the centralizer $Z$ of $F$ is a copy of $\GL_2(\R)$ embedded as 
$$
Z = \left(\begin{matrix}* & * & 0 \\
* & * & 0 \\
0 & 0 & *\end{matrix}  \right), 
$$
and we can further simplify our problem by conjugating by elements of
$Z$. %Note that $N$ acts
%transitively on $\spa (\EE_1, \EE_2) \sm \{0\}$ and preserves the
%decomposition $\R^3 = \spa (\EE_1, \EE_2) \oplus \spa (\EE_3)$. 
We decompose $\mathfrak{g}$ into eigenspaces for $\Ad(f_1)$, writing
$\mathfrak{g} = V^+ \oplus V^- \oplus V^0$, where 
$$
V^+ : = \spa(E_{13}, E_{23}), \ V^-:= \spa(E_{31}, E_{32}), \ V^0 :=
\mathfrak{z}
$$
(where $\mathfrak{z}$ is the Lie algebra of $Z$, and this is the
decomposition into eigenspaces of $\Ad(f_1)$ with eigenvalues $e^3,
e^{-3}, 1$ respectively). 
Since conjugation by $F$ preserves Haar
measure on $H$, if $\mathfrak{h}$ contains $V^+$ it also contains
$V^-$. Since $V^+$
and $V^-$ generate $\mathfrak{g}$ as a Lie algebra, this is
impossible, so 
\begin{equation}\label{eq: first restriction}
\mathfrak{h} \cap V^+ = \mathfrak{u}, \  \mathfrak{h} \cap
V^- = \mathfrak{n}.
\end{equation}
A direct computation in the
adjoint representation $\Ad: G \to \GL(\mathfrak{g})$ shows that $Z$ acts transitively on nonzero elements of
$V^+$ and also acts transitively on nonzero elements
of $V^-$. Moreover when acting on $\mathfrak{g}
\oplus \mathfrak{g}$ via $\mathrm{Ad} \oplus \mathrm{Ad}$, there is an
element of $Z$ which maps $\mathfrak{u}$ to $\spa(E_{13})$ and maps
$\mathfrak{n}$ to either $\spa(E_{31})$ or $\spa(E_{32})$. With no
loss of generality we apply such a conjugation, and treat first the
case that 
\begin{equation}\label{eq: first case}
\mathfrak{u} = \spa(E_{13}), \ 
\mathfrak{n} = \spa(E_{32}).
\end{equation}
%The action of
%$\Ad(f_1)$ on $V^+$ (respectively on $V^-$) is via the eigenvalue
%$e^3$ (resp. $e^{-3}$). 

Then $H$ contains the group $U_0$ generated by $U_{13}, U_{32}$, which
is 3-dimensional with Lie algebra $\mathfrak{u}_0 : = \spa( E_{13},
E_{32}, E_{12})$. There is no proper 
Lie subalgebra of $\mathfrak{g}$ which is $\Ad(f_1)$-invariant,
satisfies \eqref{eq: first restriction}, and
properly contains $\mathfrak{u}_0$. This implies that $H =U_0.$ But
$U_0$ is a conjugate of $U^+$, by a conjugation which leaves $F$
inside the group of diagonal matrices. By applying such a conjugation we obtain a
contradiction to Proposition \ref{prop: upper triangular} and \eqref{eq: Poincare}. 
%properly contains $U_0$ then 
%By applying a
%further conjugation
%\[
%f_t=
%\left(
%\begin{array}{ccc}
%e^t & 0 & 0 \\
%0 & e^{-2t} & 0 \\
%0 & 0 & e^{t}
%\end{array}
%\right), 
%\quad 
%U =U_{12},\quad  U_{23}\le H .
%\]
%Since the group $U_0$ generated by $U_{12}$ and $U_{23}$ is the upper triangular unipotent subgroup,
%Theorem \ref{thm;ekl} implies that 
% the measure $\mu $  is equal to $m$ if $H\neq U_0$.
% 
% We will show that it is not possible to have $H= U_0$. 
% Assume  the contrary and s
%Suppose that $H=H_{x_0}$ where $x_0=g_0\Gamma$ for some 
%$g_0\in G$. Since $g_0^{-1}Hg_0\cap \Gamma$ is a lattice in $H$, the group 
%$g_0^{-1}Hg_0$ is defined over $\mathbb Q$. 
%So both the normalizers of $H$ and $g_0^{-1}Hg_0$ are $\mathbb Q$-parabolic subgroups 
%of $G$. Therefore there exists $g\in \SL_3(\mathbb Q)$ such that 
%\[
%g_0^{-1}Hg_0=g^{-1} H g.
%\]
%It follows that $g_0=ng $ where $n\in N_G(H)$.  Therefore for any $h\in H$ we have 
%$f_t hg_0 g^{-1}{\mathbf e_1}\to 0$
%as $t\to -\infty$, where $\EE_1$ is the first element of the standard
%basis of $\R^3$. This shows that $\{f_t y : t<0\}$ is a divergent
%trajectory in $X$ for any $y \in Hx$, and this contradicts \eqref{eq: Poincare}. 
% Since $g^{-1}{\mathbf e_1}\in \mathbb Q^3$
%we have  $f_t m_x\to 0$ in the space of finite measures on $X$. 
%In view of the  choice of $x_0$
%this contradicts 
%property (\rmnum{3})  of Theorem \ref{thm;mozes}.
%{\color{blue} Try to explain this without using the conjugacy of $\mathbb Q$-parabolics.}

We now continue with the assumption $a=1$ and assume that \eqref{eq:
  first case} does not hold, so that (after conjugating by an element
of $Z$)
\begin{equation}\label{eq: second case}
\mathfrak{u} = \spa(E_{13}), \ \mathfrak{n} = \spa(E_{31}).
\end{equation}
% \[
%f_t=
%\left(
%\begin{array}{ccc}
%e^t & 0 & 0 \\
%0 & e^t & 0 \\
%0 & 0 & e^{-2t}
%\end{array}
%\right),
%\quad 
%U =U_{13},\quad \exp(E_{31}+sE_{32})\in H.
%\]
%for some $s\in \mathbb R$. 
Then $H$ contains the group $H_0 \cong \SL_2(\R)$ whose Lie algebra is
generated by $\mathfrak{u}$ and $\mathfrak{n}$, and $F \not \subset H_0$. By Proposition
\ref{prop: copy of sl2}
and \eqref{eq: Poincare} we cannot have $H=H_0$. So $H_0 \varsubsetneq
H$ and since the group generated by $F$ and $H_0$ contains the full
diagonal group, $H$ is invariant under conjugation by all elements of
the diagonal group. Therefore $H$ must contain at least one other eigenspace
$U_{ij}$ not contained in $H_0$. By \eqref{eq: first restriction}, $H$
contains one of $U_{12}, U_{21}$. However $H_0$ and any one of these
two groups generate a group which contains one of $U_{23}, U_{32}$ and
\eqref{eq: first restriction} cannot hold. 

Finally suppose $a<1$ so that the three eigenvalues of $f_1$ are
distinct. 
%
% {\bf Case 3}. 
% \[
%f_t=
%\left(
%\begin{array}{ccc}
%e^{at} & 0 & 0 \\
%0 & e^{bt} & 0 \\
%0 & 0 & e^{-t}
%\end{array}
%\right)
%\]
%where $a>b, a+b=1$ and $a,b\neq 0$.
In this case $E_{12}, E_{13}$ and $E_{23}$
belong to different eigenspaces of $\mathrm {Ad}(f_1)$, with
corresponding eigenvalues $e^{1-a}, e^{1-b}, e^{a-b}$. The equations
$a+b=-1, 0<a <1$ imply that these eigenvalues are distinct:
$$e^{1-b} > e^{a-b} > e^{1-a}.$$
Moreover the product of the eigenvalues that correspond to
eigenspaces belonging to $\mathfrak{h}$ is 1, since conjugation by elements of $F$ preserves the 
Haar measure on $H$. We consider the possibilities for $H$. 
The smallest possible value of $\dim H$ is when $H$ is generated by a pair
$U_{ij}, U_{ji}$. That is, up to a conjugation by a matrix preserving
the diagonal group, $H$ coincides with the group $H_0$ considered
above. But this leads to a contradiction via
\eqref{eq: Poincare} and Proposition \ref{prop: copy of sl2}.  

If $\dim H \geq 4$ then $H$ contains at least two expanding or two
contracting eigenvalues. 
It is easy to check that (up to re-indexing) $H$ contains
$U_{13}, U_{21}, U_{32}$, and these groups generate $G$, which is impossible. 
\end{proof}

\section{Equidistribution of a line segment}
The aim of this section is to prove Theorems \ref{thm;equiline}, \ref{thm;dirichlet}
and \ref{thm: weights} (\rmnum{1}). 
We first assume the notation and assumptions in Theorem \ref{thm;equiline},
in particular $f_t^{(\br)}$ and $u$ %and $\bar{u}$ 
are as in (\ref{eq;notation}),
and $\bar{\nu}$ is the image of $\nu$ under $s \mapsto {u}(s, as+b)x_0$. That is
\begin{equation}\label{eq;push}
\int _X\psi \, d\bar{\nu}=\int_{\R}\psi({u}(s, as+b)x_0)\, d\nu(s)
\end{equation}
for every $\psi\in C_c(X)$. 
Sometimes we need to treat the cases where $r_1=r_2$ and 
$r_1\neq r_2$ separately, so we let $f_t:=f_t^{(1/2,1/2)}$
to emphasize that we are in the the former case. 
First we show that there is no escape of
mass. 

\begin{lem}\label{lem;nonescape}
Let $\mu$ be a weak-* limit of 
\begin{equation}\label{eq;limit}
\lim_{T\to \infty}\frac{1}{T}\int_0^{T} \left({f_{t}^{(\br)}} \right)_* \bar{\nu} \, dt\quad\mbox{ as } T\to\infty.
\end{equation}
Then $\mu(X)=1$. 
\end{lem}
\begin{proof} 
It suffices to show that for each $\vre>0$ there is a compact $K_0
\subset X$ such that for all large enough $t$, 
\begin{equation}\label{eq: for no escape}
\nu (\{s \in I: f_t^{(\br)} {u}(s, as+b)x_0 \notin K_0\}) < \vre.
\end{equation}
Since $\nu$ is absolutely continuous with respect to Lebesgue measure
on $\R$, we can write $d\nu(s) = h(s)ds$ where $h$ is a
non-negative measurable function on $I$ with $\int_I h(s)ds=1$. 
Given $\vre>0$, let $R$ be large enough so that 
$$
\int_{I_R} h(s)\, ds < \frac{\vre}{2}, \ \text{ where } I_R := \{s \in I: h(s) \geq R\}.
$$
Then in order to establish \eqref{eq: for no escape}, by considering
separately $I_R$ and $I \sm I_R$, it suffices to
find a compact $K_0 \subset X$ such that for all sufficiently large
$t$, 
\begin{equation}\label{eq: what we need}
\frac{|\{s \in I: f_t^{(\br)} {u}(s, as+b)x_0 \notin K_0  \}|}{|I|} < \frac{\vre}{2R}
\end{equation}
(where $|A|$ denotes the Lebesgue measure of $A \subset \R$). 
Using \eqref{eq: weird condition}, let $K \subset X$ be a compact
subset such that for each $t$,  there is 
$s_t \in J$ with $f_t^{(\br)} {u}(s_t, as_t+b)x_0 \in K$. 
We choose $c>0$ so that $I\cup J\subset [-c, c]$. 
%By decomposing $I$ into
%two subintervals, there is no loss of generality in assuming $s_0$ is one of the endpoints of $I$, so from now
%let us assume that $I = [s_0, s_0+r]$.
%\[
%c_1=\min\{c-s_0, 0\} \mbox{ and } d_1=\max\{0, d-s_0\}.
%\]
Multiplying matrices, one sees that 
\begin{align}
 & f_t^{(\br)} {u}(s, as+b)x_0  \label{eq: one sees} \\
 = &u\left( e^{(r_1+1)t}(s-s_t) , ae^{(r_2+1)t}(s-s_t) \right) f_t^{(\br)}
{u}(s_t, as_t+b)x_0. \notag
\end{align}
By assumption \eqref{eq: weird condition}, $f_t^{(\br)}
{u}(s_t, as_t+b)x_0 \in K$ where $K \subset X$ is a compact set. 
It follows from \cite[Theorem 6.1]{dm93} that given $\vre >0$ there exists a compact subset 
$K_0$ of $X$ such that 
for every $x\in K$ and every $t\ge 0$ one has
\begin{align}
&\left| \left \{s\in [-c,c]:  u \left(e^{(r_1+1)t}(s-s_t) , ae^{(r_2+1)t}(s-s_t) \right) x\notin K_0
  \right \} \right|   \label{eq;epsilon} \\
 &< \left (\frac{\vre |I|}{4cR} \right)2c. \notag
\end{align}
Combining \eqref{eq: one sees} with \eqref{eq;epsilon} gives
\eqref{eq: what we need}. 
%one finds that
%for any $t \geq 0$, 
%$$
%|\{s \in I : f_t \bar{u}(s, as+b) \in K\}| \geq (1-\vre)|I|. 
%$$
%This implies that for any 
%In particular (\ref{eq;epsilon}) holds for $x=f_t \bar{u}(s_0, as_0+b)$. 
%Thus for any $t \ge 0$, 
%$$f_*\nu|_I(K) \geq 1-\vre$
\end{proof}

Next we show unipotent invariance. 
\begin{lem}\label{lem;invariance}
Any weak-* limit   of  (\ref{eq;limit}) is invariant under 
some one dimensional unipotent subgroup $U$ of $G$ 
normalized by $\left\{f_t^{(\br)}:t\in \mathbb R \right\}$. 
\end{lem}

\begin{proof}
To simplify the notation we let 
\[
\ell:\mathbb R\to \mathbb R^2 , \ \ell(s):=(s, as+b)^{\mathrm{tr}}.
\]
We first prove that  in the case 
$r_1=r_2$, any  limit measure of (\ref{eq;limit}) is invariant under 
$U=\{u(s, as): s\in \mathbb R\}$. 
It suffices to show that for any $\tilde s \in \R$, 
\begin{equation}\label{eq;asymp}
\lim_{t\to \infty}(f_{t})_* \bar{\nu} -(u(\tilde s,a\tilde s)f_{t})_* \bar{\nu}=0.
\end{equation}
Let $h \in L^1(\R)$ be a non-negative function such that $d\nu(s) =
h(s)ds,$ and let $\psi\in C_c(X)$. We have: 
\begin{eqnarray*}
& & \int_X \psi\, d\big[(f_{t})_* \bar{\nu} - (u(\tilde s,a\tilde s)f_{t})_*
\bar{\nu} \big] \\
& = &  \int_{\R} \big[\psi(f_t {u}(\ell(s))x_0)-\psi(u(\tilde s,a\tilde s)f_t {u}(\ell(s))x_0) \big] h(s)\, ds \\
& = &\int_{\R} 
\big[\psi(f_t {u}(\ell(s))x_0)-\psi(f_t {u}(\ell(s + e^{-3t/2}\tilde s))x_0) \big] h(s)\, ds. 
\end{eqnarray*}
By continuity of $\psi$, the integrand converges pointwise to 0 as $t
\to \infty.$ Since $h \in L^1(\R)$ and $\psi$ is
bounded, using the dominated convergence theorem we see that the limit
is zero. 
This implies (\ref{eq;asymp}).

If $r_1>r_2$ we show that 
any  limit measure is invariant under 
$U_{13}:=\{u(s, 0): s\in \mathbb R\}$. 
It suffices to show that for any $\tilde s \in \R$, 
\begin{equation}\label{eq;asymp2}
\lim_{t\to \infty}\left(f_{t}^{(\br)}\right)_* \bar{\nu} -\left(u(\tilde s,0)f_{t}^{(\br)}\right)_* \bar{\nu}=0.
\end{equation}
Let $\psi, h$ be as above; set
$s': = s+e^{-(1+r_1)t}\tilde s$ and compute as follows: 
\[
\begin{split}
& \int_X \psi\, d\left[\left(f^{(\br)}_{t}\right)_*\bar {\nu} - \left(u(\tilde s,0)f^{(\br)}_{t}\right)_*
\bar{\nu} \right] \\
 = &  \int_{\R} \left[\psi \left(f^{(\br)}_t {u}(\ell(s))x_0 \right)-\psi\left(u(\tilde s,0)f^{(\br)}_t {u}(\ell(s))x_0\right) \right] \, d\nu(s) \\
 = &\int_{\R} 
\left[\psi \left(f^{(\br)}_t {u}(\ell(s))x_0 \right )-\psi \left(f^{(\br)}_t {u}( e^{-(1+r_1)t}\tilde s,
0)u(\ell(s))x_0\right) \right] \, d\nu(s) \\
 = & \int_{\R} \left[\psi \left( f_t^{(\br)} {u} (\ell(s))x_0
\right )- \psi \left (f_t^{(\br)} {u}(\ell(s')) x_0\right)
\right] \,  d\nu(s)+ \\ 
 & \int_{\R} \left[ \psi \left(f_t^{(\br)} {u}(\ell(s'))x_0
\right )- \psi \left(f_t^{(\br)} u(0, -ae^{-(1+r_1)t}) {u}(
\ell(s')) x_0
\right) \right]  d\nu(s)
. 
\end{split}
\]
By a change of variables, the absolute value of the first summand in this integral is bounded
above by $2 \sup |\psi| \int_{\R} |h(s)-h(s')|ds$, which tends to
zero as $t \to +\infty$ since $s' \to s$ and the regular
representation of $\R$ on $L^1$ is continuous. 

To bound the second summand we argue as follows. 
\[
\begin{split}
&  \int_{\R} \left[ \psi \left(f_t^{(\br)} {u}(\ell(s'))
x_0\right )- \psi \left(f_t^{(\br)} u(0, ae^{-(1+r_1)t}){u}
(
\ell(s')) x_0
\right) \right] \, d\nu(s)  \\ 
=& 
\int_{\R} \left[ 
\psi \left( 
f_t^{(\br)} {u}(\ell(s'))x_0
\right) - \psi \left( u(0, as_0e^{(r_2-r_1)t}) f_t^{(\br)} {u}(\ell(s'))x_0
\right) 
\right ] \, d\nu(s),
\end{split} 
\]
and this tends to zero by the uniform continuity of $\psi$ and the dominated convergence
theorem.
% as in the proof ofLemma \ref{lem;invariance}. 
Hence $\left(f^{(\br)}_{t}\right)_* \bar{\nu}  -
\left(\exp(s_0 E_{13}) f^{(\br)}_t \right)_* \bar{\nu} \to_{t
  \to \infty} 0$. Since $\mu$ is a sequential limit as $T \to \infty$,
we see that $\mu$ is $U_{13}$-invariant, as required. 

Finally we consider the case where $r_1<r_2$. If $a\neq 0$ then 
a similar argument as for the case where $r_1>r_2$ implies 
the invariance for $U_{23}$. If $a=0$ then 
the argument for the case where $r_1=r_2$ goes through and shows 
that the limit measure is invariant under $U_{13}$. 
\end{proof}

\begin{prop}\label{prop;dynamical2}
Let $\lambda$ be a probability measure on $\R^2$.  Suppose that  
%the measure $\bar{\nu} = \bar{u}_* \lambda$ satisfies
\begin{equation}\label{eq: assumption lambda}
\frac1T \int_0^T \left(f^{(\br)}_{t} \bar u\right)_*\lambda
\, dt \to_{T \to \infty} m.
\end{equation}
Then $\lambda$-almost every $\bv \in \R^2$
admits no improvement in Dirichlet's theorem w.r.t. weights $\br$.  
\end{prop}

\begin{proof}
 According to \cite[Prop. 2.1]{KW_di}, if $\left\{f^{(\br)}_t \bar{u}(\bv): t
 \geq 0 \right\}$ is dense in $X$ then $\bv$ admits no improvement in
 Dirichlet's theorem w.r.t. weights $\br$. Suppose by contradiction
 that 
$$
\lambda \left( \left\{ \bv : \left\{ f^{(\br)}_t \bar{u}(\bv): t
 \geq 0\right\} \text{ is not dense} \right\} \right)>0.
$$
Let $\{U_1, U_2, \ldots\}$ be a countable collection of open subsets of $X$ which form a basis for
the topology of $X$. Then for some $i$, 
$$
\lambda (A)> 0, \ \text{ where } A: = \left\{ \bv : \forall t \geq 0, \,  f^{(\br)}_t
    \bar{u}(\bv) \notin U_i  \right\}.
$$
Let $\lambda_0$ be the (normalized) restriction of $\lambda$ to $A$,
let $\lambda_1$ be the (normalized) restriction of $\lambda$ to the
complement of $A$, and choose a sequence $\{T_n\}$ with $T_n\to \infty$ such that 
$$\mu_0:=\lim_{n \to \infty} \frac{1}{T_n} \int_0^{T_n}
\left(f^{(\br)}_{t}\bar u \right)_*\lambda_0 \,
dt$$ 
exists. 
Then $\mu_0$
 gives zero mass to $U_i$. In view of \eqref{eq: assumption lambda},
 the limit $\mu_1 = \lim_{n \to \infty} \frac{1}{T_n} \int_0^{T_n}
 \left(f^{(\br)}_{t}\bar u \right)_*\lambda_1 \,
dt$ also exists, and $m$ is a convex combination of $\mu_0$ and
$\mu_1$ with weights $\lambda(A), \, 1-\lambda(A)$. Both measures
$\mu_0, \mu_1$ are invariant under $\left\{f^{(\br)}_t\right\}$, and since $m$ is
ergodic, $m = \mu_0=\mu_1$. This contradicts the fact that
$\mu_0(U_i)=0$. 
 \end{proof}

\begin{proof}[Proof of Theorem \ref{thm;equiline}]
Let $\mu$ be a weak-* limit of  (\ref{eq;limit}).
Then $\mu$ is invariant under the one parameter diagonal subgroup 
$F:=\left\{f_t^{(\br)}:t\in \mathbb R \right\}$. It follows from Lemma \ref{lem;invariance} that
$\mu$ is also invariant under some one-parameter unipotent group $U$
normalized by $F$. 
Lemma \ref{lem;nonescape} implies that $\mu$ is a probability measure. Therefore 
$\mu=m$ according to Theorem \ref{thm;measure}. Since $\mu$ is an arbitrary weak-* limit
as $T\to \infty$, the conclusion follows. 
\end{proof}

 \begin{proof}[Proof of Theorem \ref{thm;dirichlet} and \ref{thm: weights}(\rmnum{1})]
We only prove the latter since the former is a special case. 
By switching the roles of $x$ and $y$ there is no loss of generality
in assuming that $\mathcal{L}$ is not vertical, i.e. it is given by an
equation of the form $s \mapsto \ell(s) := (s, as+b)$ for some $a, b \in \R$. 
Let
$\tilde s\in \mathbb R$ such that $\ell(\tilde s)$ is badly approximable
w.r.t.~weights $\br$. According to Dani's
correspondence \cite{d85}, and its generalization to the framework of
approximation with weights \cite{kleinbock duke}, there is a compact $K \subset X$ such that
$f_t ^{(\br)}\bar{u}(\ell(\tilde s)) \in K$ for all $t \geq0$. That is, \eqref{eq:
  weird condition} is satisfied. Now the conclusion is 
immediate from 
% Without loss of generality we assume that the line $\mathcal L$ is
% given by the equation $\mathrm{y}=a\mathrm{x}+b$.
% Suppose that $(s_0, as_0+b)$ is a badly approximable vector on $\mathcal L$. 
%  It follows from Dani's correspondence \cite{d85} that there exists a compact subset $K_0$ of $X$ such that 
%\[
%f_t u(s_0, as_0+b)\Gamma\in K_0
%\]
%for every $t\ge 0$. 
%Let $I$ be an arbitrary interval of positive length and $\lambda$ be the normalized (probability) volume measure of $\mathcal L|_I$.
%It follows from Theorem \ref{thm;equiline}  that 
%(\ref{eq;assumption}) holds for $\nu$ defined in (\ref{eq;lift}).
%So 
Theorem \ref{thm;equiline} and 
Proposition \ref{prop;dynamical2}. % implies that
%for almost every $s\in I$ the DT cannot be  improved for $(s, as+b)$.
%Since $I $ is arbitrary the conclusion follows. 
 \end{proof}

\section{Equidistribution of a nondegenerate curve}\label{sec: curve}
The goal of this section is to prove Theorems \ref{prop;curve}, 
\ref{prop;equicurve} and \ref{thm: weights}(ii). Our argument
uses many ideas of Shah \cite{s09, s10} but is made significantly simpler
by the extra averaging with respect to $t$, appearing in
Proposition  \ref{prop;dynamical2}. 

Let the
notation be as in Theorem \ref{prop;equicurve}. 
We write $f_t=f_t^{(1/2,1/2)}$ and  $\varphi=(\varphi_1, \varphi_2)$ where each
$\varphi_i$ is a $C^2$ function on $[0,1]$.
Without loss of generality we further assume that $r_1\ge r_2$. 
We claim that $\varphi'_1(s) \neq 0$ for a.e. $s$; indeed, set 
$$
A:= \{s \in [0,1]: \varphi'_1(s)=0\}
$$
and let $A'$ denote the set of Lebesgue density points of $A$. Then
$A$ and $A'$ have the same Lebesgue measure, and by Rolle's
theorem, for $s \in A'$, 
$$\varphi_1'(s)=\varphi_1''(s)=0.$$
Thus the Wronskian
determinant of $\varphi'$ vanishes on $A'$, so by nondegeneracy $A$ and $A'$ must have
measure zero. 

 It follows that there exists a countable 
collection $\mathcal I$ of closed intervals such that 
\begin{itemize}
\item $\cup _{\mathcal I} I$ has full measure in $[0,1]$ and $I_1\cap I_2$ contains at most one point for distinct 
$I_1, I_2\in \mathcal I$.

\item $\varphi'_1(s)\neq 0$ for every $s\in \bigcup_{I\in \mathcal
    I}I^\circ$ (where $I^\circ$ is the interior of $I$).
\end{itemize}

Therefore it suffices to prove Theorem \ref{prop;equicurve} for each closed interval
properly contained in some
 $I\in \mathcal I$, replacing $\nu$ with the restriction of $\nu$ to
 this closed interval. So we assume without loss of generality that $\varphi_1'(s)\neq 0$
 for every  
 $s\in [0,1]$. 
 
  There exists a continuously differentiable function $M: [0,1]\to \SL_2(\mathbb R)$ such that 
 $M(s)\varphi'(s)=\mathbf e_1$. We define the map 
 \[
 z:[0,1]\to \SL_3(\mathbb R)\quad \mbox{by } z(s)=
 \left(
 \begin{array}{cc}
 M(s) & 0 \\
 0 & 1
 \end{array}
 \right).
 \]
 Let $\nu_\varphi$ be the probability measure on $X$ defined by 
\begin{equation}\label{eq;pushcurve}
\int _X\psi \, d\nu_\varphi= \int \psi(z(s)\bar{u}(\varphi(s))) \, d\nu(s)
\end{equation}
for every $\psi\in C_c(X)$. 
We set 
\[
\nu_{\br}:=\left\{
\begin{array}{ll}
\nu_\varphi & \quad \mbox{if } r_1=r_2\\
(\bar{u})_*\nu & \quad \mbox{if } r_1> r_2.
\end{array}
\right.
\]

\begin{lem}\label{lem;cunipotent}
Any weak-* limit of 
\begin{equation}\label{eq;climit}
\frac{1}{T}\int_0^T \left(f_t^{(\br)} \right)_*\nu_\br \, dt \quad \mbox{as } T\to \infty 
\end{equation}
is invariant under the group $U_{13}=\{u(s,0): s\in \mathbb R\}.$
\end{lem}
\begin{proof}
In the case where $r_1=r_2$
it suffices to prove that for any $\psi\in C_c(X)$, any $\vre>0$, and any $\tilde s\in \mathbb R$,
\begin{equation}\label{eq;goalinvariant}
\left|\int_0^1 \Big [\psi(f_tz(s)\bar{u}(\varphi(s)))
-\psi(u(\tilde s,0)f_tz(s)\bar{u}(\varphi(s)))\Big ]\, d\nu(s)
\right|<\vre
\end{equation}
provided that $t$ is sufficiently large.

We fix a $C^2$ extension of $\varphi$ on $[-1,2]$.
On the one hand, a change of variables, the boundedness of $\psi$, and
the continuity of the regular representation
imply that  
%\begin{equation}
%\label{eq;firsthalf1}
$$
\int_0^1\Big|\psi(f_tz(s)\bar{u}(\varphi(s)))-\psi(f_tz(s)\bar{u}(\varphi(s+\tilde
s e^{-3t/2})))\Big|\, d\nu(s)  \notag
\to_{t \to \infty} 0. 
$$
%\end{equation}
On the other hand,  since $\varphi$ is a $C^2$-function on a compact interval,  
\[
\varphi(s+\tilde s e^{-3t/2})=\varphi(s)+\tilde se^{-3t/2} \varphi'(s)+O(e^{-3t})
\ \text{ as }  t \to +\infty,
\]
where the implicit constant in the error term is independent of
$s$. Therefore
\begin{align}
& %\int_0^1     
f_tz(s)\bar{u}(\varphi(s+\tilde s e^{-3t/2}))
%\, ds 
\label{eq: commuting matrices}
\\ 
= & %\int_0^1 
f_t z(s) u\left[\varphi(s)+\tilde se^{-3t/2} \varphi'(s) +
  O\left(e^{-3t} \right) \right]\pi(e)  %\, ds 
\notag \\
 = & 
\left[  f_tz(s) u\left(\tilde s e^{-3t/2} \varphi'(s) + O(e^{-3t}) \right)
\left(f_tz(s)\right)^{-1}  \right] \left [f(t)z(s) \bar u(\varphi(s))\right] 
\notag \\
 = & 
u\left(\tilde sE_{13}  +O( e^{-3t/2} ) 
\right)
f_t z(s) \bar u(\varphi(s)) 
 \notag \\
 = & u(O\left(e^{-3t/2}) \right) u(\tilde s, 0) f_t z(s) \bar{u}(\varphi(s)).\notag
\end{align}
By uniform continuity of $\psi$, this implies that 
%\begin{equation}\label{eq;secondhalf}
\begin{align*}
&\int_0^1 \psi \left(f_tz(s)\bar{u}(\varphi(s+\tilde s e^{-3t/2})) \right) 
\, d\nu(s)   \\
\to &\int_0^1 \psi \left(u(\tilde s, 0) f_t z(s) \bar{u}(\varphi(s))
\right) \, d\nu(s) 
%\end{equation}
\end{align*}
as $t \to +\infty$.  
Now \eqref{eq;goalinvariant} follows
%from
%\eqref{eq;firsthalf2} and \eqref{eq;secondhalf}, 
for all large enough $t$. 

In the case where $r_1> r_2$  it suffices to show that 
for any $\psi\in C_c(X)$, any $\vre>0$, and any $\tilde s\in \mathbb R$,
\begin{equation}\label{eq;goalinvariant2}
\left|\int_0^1 \Big [\psi\left(f_t^{(\br)}\bar{u}(\varphi(s)) \right)
-\psi(u(\tilde s,0)f_t^{(\br)}\bar{u}(\varphi(s)))\Big ]\, d\nu(s)
\right|<\vre
\end{equation}
provided that $t$ is sufficiently large.

We first prove (\ref{eq;goalinvariant2}) for
$d\nu=ds$.
Let $N_t=[\delta e^{(1+r_1)t}]\in \mathbb N$
where 
\begin{equation}\label{eq;delta}
\delta=\varepsilon(16 \|\psi\|_{\mathrm {sup}}\|1/\varphi_1'\|_{\mathrm{sup}})^{-1}.
\end{equation}
Here 
\[
\|\psi\|_{\mathrm {sup}}:=\sup_{x\in X}|\psi(x)|\quad \mbox{and}\quad
\|1/\varphi_1'\|_{\mathrm{sup}}=\sup_{s\in [0,1]}|1/\varphi_1'(s)|.
\] 
In what follows we always assume
$t$ is large so that $N_t >1$.
We partition $I=\bigcup_{k=1}^{N_t}I_k$ where $I_k=[s_k,s_{k+1}]$
and $s_{k+1}-s_k=1/N_t$. 
Let 
\[
\ell_k(s)=\varphi(s_k)+(s-s_k)\varphi'(s_k).
\]
Then for all $s \in I_k$ we have 
\[
\varphi(s)=\ell_k(s)+O\left(N_t^{-2}\right)
\]
and, arguing as in \eqref{eq: commuting matrices}, 
$$f_t^{(\br)} \bar{u}(\ell_k(s)) = u(O(N_t^{-1}))f_t^{(\br)}
\bar{u}(\varphi(s)).$$ 
Therefore for $t$ sufficiently large  we have 
\begin{eqnarray*}
\left |\int_0^1\psi\left (f_t^{(\br)}\bar{u}(\varphi(s))\right)\, ds-
\sum_{k=1}^{N_t}\int_{I_k}\psi\left (f_t^{(\br)}\bar{u}(\ell_k(s))\right)ds
\right|
\le \frac{\varepsilon}{4}.
\end{eqnarray*}
The same holds for $\psi(u(\tilde s)\cdot)$ in place of $\psi$. 
Therefore to prove (\ref{eq;goalinvariant2}) it suffices to show that 
for $t$ sufficiently large 
\begin{equation}\label{eq;newgoal2}
\sum_{k=1}^{N_t}\int_{I_k}
\left|\psi\left (f_t^{(\br)}\bar{u}(\ell_k(s))\right)
-\psi\left (u(\tilde s, 0)f_t^{(\br)}\bar{u}(\ell_k(s))\right)
\right|
\, ds<\frac{\varepsilon}{2}.
\end{equation}
For $1\le k\le N_t$ let $\tilde s_k=\tilde se^{-(1+r_1)t)}\varphi_1'(s_k)^{-1}$. 
We have 
\begin{eqnarray}
& & u(\tilde s, 0)f_t^{(\br)}\bar{u}(\ell_k(s))  \notag \\
& =&  f_t^{(\br)}u(0, -\tilde s_k\varphi_2'(s_k)) \bar{u}(\ell_k(s+\tilde s_k)) \notag \\
& = & u(0,-\tilde s_k e^{(1+r_2)t}) f_t^{(\br) }   \bar{u}(\ell_k(s+\tilde s_k)). \label{eq;newgoal3}
\end{eqnarray}
By dominated convergence theorem and (\ref{eq;newgoal3}), to prove (\ref{eq;newgoal2})
it suffices to show that for $t$ sufficiently large
\begin{equation}\label{eq;newgoal4}
\sum_{k=1}^{N_t}\int_{I_k}
\left|\psi\left (f_t^{(\br)}\bar{u}(\ell_k(s))\right)
-\psi\left (f_t^{(\br)}\bar{u}(\ell_k(s)+\tilde s_k)\right)
\right|
\, ds<\frac{\varepsilon}{4}.
\end{equation}
The left hand side of (\ref{eq;newgoal4}) is 
\[
\le N_t(2\|\psi\|_{\mathrm{sup}}\tilde s e^{-(1+r_1)t}\|1/\varphi_1'\|_{\mathrm{sup}})\le \varepsilon/4
\]
by (\ref{eq;delta}) as required. 

Now we turn to the proof of (\ref{eq;goalinvariant2}) for general $\nu$. 
We write $\nu=h(s)\,ds$
for some nonnegative  function $h$ on $[0,1]$. 
The case for $\nu=ds$ implies the case where $h$
is a characteristic function of open subsets. 
By approximating functions in $L^1$ norm we get the results 
for characteristic functions and finally for any $h$. 
\end{proof}

\begin{lem}\label{lem;cnonescape}
Any weak-* limit of (\ref{eq;climit}) is a probability measure. 
\end{lem}
\begin{proof}
Since $z([0,1])$ is relatively compact, it suffices to prove 
no escape of mass replacing $\nu_\br$ by $(\bar{u})_*\nu$.
As in the proof of Lemma \ref{lem;nonescape}, we can reduce the
problem to the case that $\nu$ is the measure $ds$; then one uses \cite[Proposition
2.3]{km98}. 
\end{proof}

\begin{lem}\label{lem;equitwist}
We have
\[
\frac{1}{T}\int_0^T \left(f_t^{(\br)}\right)_*\nu_{\br} \, dt\to_{T \to
  \infty}  m.
\]
\end{lem}
\begin{proof}
Let $\mu$ be a weak-* limit of (\ref{eq;climit}). It is easy to see that $\mu$ is invariant under 
$F:=\{f_t: t\in \mathbb R\}$. It follows from Lemma \ref{lem;cunipotent} that $\mu$ is invariant 
under the group $U_{13}$. In view of Lemma \ref{lem;cnonescape} the measure $\mu$ is a probability 
measure. Therefore Theorem \ref{thm;measure} implies that $\mu=m$. Since $\mu$
is an arbitrary weak-* limit, the conclusion follows. 
\end{proof}

\begin{proof}[Proof of Theorem \ref{prop;equicurve}]
If $r_1\neq r_2$, then the conclusion is contained in Lemma \ref{lem;equitwist}.
Now we prove the case where $r_1=r_2=1/2$.
  It suffices to show that given 
  $\psi\in C_c(X)$ and $\vre>0$
   one has
\begin{equation}\label{eq;cgoal}
\left |\frac{1}{T}\int_0^T\int_0^1\psi(f_t\bar{u}(\varphi(s)))\,
  d\nu(s)dt-\int_X \psi \, dm
\right |<\vre
\end{equation}
 for $T$ sufficiently large.
We first divide $[0,1]$ into finitely many closed intervals $\{I_k: 1\le k\le N\}$ such that for any
points $s, \tilde s\in I_k$ and any $x\in X$ one has 
%{\color{blue} I made a  slight change here, check that you agree.}
\begin{equation}\label{eq;uniform}
|\psi(z(\tilde s)^{-1}z(s)x)-\psi(x)|<\frac{\vre}{2}.
\end{equation}
Let $s_k$ be the left endpoint of the interval $I_k$.
Since the matrices $z(s)$ commute with $f_t$, we have 
\begin{eqnarray}\label{eq;rewrite}
 & & \frac{1}{T}\int_0^T\int_0^1\psi\left(f_t\bar{u}(\varphi(s))\right)\, d\nu(s) dt  \\
& = & \sum_{k=1}^N\frac{1}{T}\int_0^T\int_{I_k}\psi \left(z(s)^{-1}z(s_k)z(s_k)^{-1}f_tz(s)
\bar{u}(\varphi(s))\right)\, d\nu(s)  dt. \notag
\end{eqnarray}
In view of (\ref{eq;uniform}) and  (\ref{eq;rewrite}) to prove   (\ref{eq;cgoal}) it suffices to show
that for $T$ sufficiently large
\begin{align*}
&\left |
\frac{1}{T}\int_0^T\int_{I_k}\psi\left(z(s_k)^{-1}f_tz(s)
\bar{u}(\varphi(s))\right)\, d\nu(s) dt\right.  \\
&\left.
-|I_k| \int_X \psi\left(z(s_k)^{-1}x\right)\, dm \right|
<\frac{\vre}{2}.
\end{align*}
This follows from Lemma \ref{lem;equitwist} applied to the function $x
\mapsto \psi(z(s_k)x).$
\end{proof}

 \begin{proof}[Proof of Theorem \ref{prop;curve} and \ref{thm: weights}(ii)]
 Follows from Theorem \ref{prop;equicurve} and  Proposition \ref{prop;dynamical2}.
 \end{proof}

\section{Some examples}\label{sec: examples}
In this section we give some examples which explain the necessity of
conditions which appear in our theorems.

\subsection{Examples for Theorem \ref{thm;measure}}\label{subsec: elon}
All of the conditions of Theorem \ref{thm;measure} are necessary for
its validity. The following examples illustrate 
two of them which are not obvious to see. 

First we show that the assumption that $F$ has no nonzero invariant vectors in $\mathbb R^3$
is necessary. We can embed $\SL_2(\mathbb R)\ltimes \mathbb R^2$ into $G$ so that it 
induces an embedding of 
 \[
 Y=\left(\SL_2(\mathbb R)\ltimes \mathbb R^2\right)/ \left(\SL_2(\mathbb Z)\ltimes \mathbb
 Z^2 \right)
%\cong 
 % \SL_2(\mathbb R)/\SL_2(\mathbb Z)\times \mathbb R^2/\mathbb Z^2
\]
into $X$. An example of such an embedding is the map $\tau$ which sends 
$(g, \bf v)$  to 
$
\left (
\begin{array}{cc}
g & \bf v \\
0 & 1
\end{array}
\right)
$
where $g\in \SL_2(\mathbb R) \mbox{ and } \bf v\in\mathbb R^2$.
Let $\mu_1$ be the standard probability measure on $Y$ induced by the
haar measure on $\SL_2(\R) \ltimes \R^2$ and let $\mu$ be its image
under the map above. Then $\mu$ is clearly invariant under the group
$F': = \tau(F) $ and also under $U' := \tau (\{ (I_2,
(s,0)^{\mathrm{tr}}): s\in \R\},$ where $I_2$ is the identity in
$\SL_2(\R)$. Then $F'$ normalizes $U'$, $F'U'$ is not abelian, and the
conclusion of Theorem \ref{thm;measure} does not hold, as the
existence of $\mu$ shows.   
%According to Theorem \ref{thm;mozes} any $F'U'$ invariant probability measure on $Y$
%is invariant under $I_2\times \mathbb R^2$. Therefore $\mu_1\times \mu_2$ is ergodic 
%for  the action of the  group $F'U'$. 

In fact there are $F'U'$-invariant ergodic  measures on $X$ which are
not even
homogeneous. Indeed, it is well know that there are uncountably many
$F'$ invariant and ergodic 
 nonhomogeneous 
probability measures on $\SL_2(\mathbb R)/\SL_2(\mathbb Z)$. For each
such measure $\nu$, integrating along the fiber of $Y \to
\SL_2(\R)/\SL_2(\Z)$ constructs a measure $\nu'$ on $Y$ which is not
homogeneous. The image of any such measure under $\tau$ will be a
measure on $X$ which is $F'U'$-invariant and not homogeneous. 

Next we show that 
the theorem is not true for 
$X_4:=\SL_4(\mathbb R)/\SL_4(\mathbb Z)$. We are grateful to Elon
Lindenstrauss for pointing out this example, which
relies on some results of \cite{lw}. 
Let 
\begin{equation}\label{eq: def H'}
H':=\left(
\begin{matrix} * & * & 0 & 0 \\ * & * & 0 & 0 \\ 0& 0& * & * \\ 0& 0& * & *
\end{matrix}
\right) \subset \SL_4(\R).  
\end{equation}
In \cite{lw} it was shown, using number fields of degree 4 containing
 subfields of degree 2, how to find $x\in X_4$ such 
that $H'x$ is closed and admits a finite $H'$-invariant measure
$m'$. Let
\[
F:=\{\mathrm{diag}(e^{3t}, e^t,e^{-t}, e^{-3t})\} \quad \mbox{and}\quad
U: =U_{12}.
\]
Then clearly $F, U$ satisfy the conditions of Theorem
\ref{thm;measure}, and $m'$ is $FU$-invariant but not
$\SL_4(\R)$-invariant. 

%In fact it can be shown that in this example $m'$ is $FU$-ergodic. To
%see this, let $H_0 \cong \SL_2(\R) \times \SL_2(\R)$ be the maximal
%semisimple subgroup of $H$. Then (up to finite covers) $H_0$ is isomorphic to the quotient
%of $H'$ by its center, $H_0x$ is closed and of finite-volume,
%and the natural map $H'x \to H_0x$ induced by the homomorphism $H' \to
%H_0$ maps $m'$ to the $H_0$-invariant measure
%$m_0$ on $H_0x$. 
%In these examples the stabilizer $(H_0)_x$ of $x$ in $H_0$ is an
%irreducible lattice in $H_0$. 
%Thus by the Howe-Moore theorem $U$ acts
%ergodically on $m_0$, and since $F$ normalizes $U$, $U$ acts
%ergodically on each $f_*m_0, f \in F$. Thus for any $FU$-invariant
%measurable function $\varphi \in L^1(X, m')$, for every $f \in F$ the
%restriction $\varphi|_{fH_0x}$ is $f_*m_0$-almost everywhere constant. Since $F$ projects onto
%$H'/H_0$, $\varphi$ is $m'$-a.e. constant. Hence $m'$ is ergodic for the action of $FU$. 

\subsection{Example for Theorem \ref{thm;dirichlet}}
The goal of this subsection is to show that Theorem
\ref{thm;dirichlet} does not extend to $n=3$. That is, we prove:
\begin{thm}\label{thm: example for dirichlet}
There is a line segment  $\mathcal{L} \subset \R^3$ which contains a badly approximable vector,
such that every point in  $\mathcal L$ admits an
improvement in Dirichlet's theorem. 
\end{thm}

The proof is an elaboration on the construction in \S
\ref{subsec: elon}, and also uses a result of Haj\'os, which we now
state. For a permutation $\sigma $ of $\{1, \ldots, n\}$, let
$U_{\sigma}^+$ denote the group generated by $\{U_{\sigma(i)
  \sigma(j)} : i<j\}$; that is the conjugate of the upper triangular
group by the permutation matrix corresponding to $\sigma$. 

\begin{thm}[Haj\'os]\label{thm: hajos} 
Let $X_n$ be the space of unimodular lattices in $\R^n$ and let
$\Lambda \in X$ such that $\Lambda$ contains no nonzero points in the
interior of the unit cube. Then there is $\sigma$ such that $\Lambda
\in U^+_\sigma \Z^n$.
\end{thm}

Note that each of the orbits $U^+_\sigma \Z^n$ is compact; thus,
recalling that $\| \cdot \|$ denotes the sup-norm, if we set  
$$K_{\vre} := \{\Lambda \in X_n : \forall v \in \Lambda \sm \{0\}, \| v
\| \geq \vre\}$$
then Theorem \ref{thm: hajos} says that $K_1$ is a finite union of compact
orbits of the groups $U^+_{\sigma}$. %These orbits have dimension $\dim
                                %U_\sigma^+ = n(n-1)/2$.  

We will also need \cite[Prop. 2.1]{KW_di}. We extend the notation
\eqref{eq;notation} and (\ref{eq;u}) to arbitrary dimension $n \geq 2$ in the obvious way. 
\begin{prop}\label{prop: hajos di}
The vector $\bv \in \R^n$ admits no improvement in Dirichlet's theorem
if and only if there is $t_n \to \infty$ such that
$\lim_{n \to \infty} f_{t_n}\bar{u}(\bv)$ exists and belongs to $K_1$. 
\end{prop} 

Let $G=\SL_4(\R),\,  X=X_4$, $H=H'$ as in \eqref{eq: def H'}
and $\pi: G\to X$ be the natural quotient map. In \cite{lw} it was shown that
there are $x \in X$ for which $Hx$ is a closed orbit of finite
volume. We will need the following well-known strengthening: 
\begin{prop}\label{prop: dense bounded}
There is a dense set of $x \in X$ such that $Hx$ is closed of finite
volume, and $\{f_tx: t \geq 0\}$ is bounded. 
\end{prop}
\begin{proof}
As shown in \cite{lw}, there are $x_0 \in X$ for which $Hx_0$ is closed
and $Ax_0$ is compact, where $A$ is the group of diagonal matrices in
$G$. Thus $x_0$ clearly satisfies the required conclusions. Now write
$x_0 = \pi(g_0)$ and let $g \in G(\Q), x:= \pi(g_0g)$. The set of such
$x$ is dense since $G(\Q)$ is dense in $G$, and we claim that
$x$ also satisfies the required conclusions; equivalently, if we
set $\Gamma = \SL_4(\Z), \, \Gamma' := g\Gamma g^{-1}$, that
$Hg_0\Gamma'$ and $\{f_tg_0\Gamma': t\geq 0\}$ are bounded in
$G/\Gamma'$. 
Since $g$ is in the
commensurator of $\Gamma$, there is a finite-index
subgroup $\Gamma_0 $ of $\Gamma$ such that the maps $\tau_1: G/\Gamma_0 \to
G/\Gamma$, $\tau_2: G/\Gamma_0 \to G/\Gamma'$ are
$G$-equivariant and proper. Since $x \in
\tau_2(\tau_1^{-1}(x_0))$, the conclusion follows. 
\end{proof} 

\begin{proof}[Proof of Theorem \ref{thm: example for dirichlet}]
Let 
\[
P : = \left(\begin{matrix} 
* & * & * & 0 \\
* & * & * & 0 \\
* & * & * & 0 \\
* & * & * & *
\end{matrix} \right) \subset G. 
\]
Then 
\begin{equation}\label{eq: prop of P}
P = \big\{p \in G: \{f_t p f_{-t}: t \geq 0 \} \text{ is bounded in }
G \big \}. 
\end{equation}
This implies that if $p \in P$ and $x \in X$ then for $t \geq 0$, the
distance between  $f_t px$ and $f_tx $ is bounded (independently of
$t$). 
Also let 
\[
Q : = \left(\begin{matrix} 
* & * & * & 0 \\
* & * & * & 0 \\
* & * & * & 0 \\
0 & 0 & 0 & *
\end{matrix} \right) \cong \GL_3(\R) \subset G. 
\]
There is a projection $q: P \to Q$ obtained by identifying $Q$ with
the quotient of $P$ by its unipotent radical, or more concretely, by
replacing the $(41), (42), (43)$ matrix entries by 0. A simple
calculation in matrix conjugation shows that for all $p \in P$, 
\begin{equation}\label{eq: prop of P0}
q(p)  = \lim_{t \to +\infty} f_t p f_{-t}. 
\end{equation}
%Similarly we let $\bar{q}: P \to \GL_3(\R)$ be the homomorphism which
%maps $(p_{ij})_{i,j \leq 4}$ to the top right $3 \times 3$ block
%$(p_{ij})_{i,j \leq 3}$. 
Let 
$$U =
\{u(\bv) : \bv \in \R^3\}  = \langle U_{14}, U_{24},
U_{34} \rangle \cong \R^3.$$
Then the set $PU$ is open and dense in
$G$. Let 
$$\mathcal{D} : = \{ g \in PU: H\pi(g) \text{  is closed, } 
\{f_t \pi(g) : t
\geq 0\}
\text{ bounded} \}.
$$
According to Proposition \ref{prop: dense bounded}, $\mathcal{D}$  is
dense in $PU$. 
Let 
$$
g = p{u}(\bv_0)\in PU
$$
for some $\bv_0 \in \R^3$ and $p \in P$.
If $g \in \mathcal{D} $ then \eqref{eq: prop of P} implies that 
$\{f_t \pi(g) : t \geq 0\}$ and $\{f_t \pi(u(\bv_0)): t \geq 0\}$ are both bounded and hence $\bv_0$ is badly
approximable. 
Now define $u_s = \exp(sE_{34}) \in H \cap U$ and consider the formula
\begin{equation}\label{eq: to be solved}
u_sp = p(s)^{-1} \til u (s).
\end{equation}
Note that $ p(s), \til u(s)$ depend on $p$ and hence on $g$ but we omit this
dependence to simplify notation.

We will show that there is $g \in \mathcal{D}$, and an open interval $I$ containing 0 such that: 
\begin{itemize}
\item[(i)] 
For all $s \in I$,
\eqref{eq: to be solved} has unique solutions $p(s) \in P$, $\til u(s) \in
U$. 
\item[(ii)] 
There is $\bw \in \R^3 \sm \{0\}$ such that $\til u(s) =
u(\tau(s) \bw)$, where $\tau(s)$ is a non-constant rational function
of $s$; that is $\mathcal{L}_0 = \{u^{-1}
\circ \til u(s) : s \in I\}$ is a
smooth parameterization of a line segment in $\R^3$. 
\item[(iii)]
For any 
%$\bar{s} \in I$ there is an interval $J \subset I$ containing
%$\bar{s}$ such that for all 
$s \in I \sm \{0\}$, $K_1 \cap q(s)Hx =
\varnothing,$ where 
$$
q(s) := q(p(s)). 
$$ 
\item[(iv)]
For any $s \in I$ such that $K_1 \cap q(s)Hx =
\varnothing$, there is no $t_n \to \infty$ for which the sequence
$\left(f_{t_n}  \til u(s) \bar u(\mathbf v_0)\right)_{ n \in \N}$ converges to an element of $K_1$. 
\end{itemize} 

First we explain why the theorem follows from
(i--iv). Consider 
$$\mathcal{L} :  = \bv_0+\mathcal{L}_0 = \{\ell(s) : s \in 
I\}, \ \text{ where } \ell(s) := \bv_0 +
\tau(s) \bw.$$ 
 According to (i), (ii) this is a nontrivial line segment in
$\R^3$, and we need to show that $\ell(s)$ admits an improvement in
Dirichlet's theorem for every $s \in I$. For $s=0$, this follows from
the fact that $\ell(0) = \bv_0$ is badly approximable using \cite{ds}. 
By (iii), for all $s \in I \sm \{0\}$ we have $K_1 \cap q(s)Hx = \varnothing.$
Then, according to (iv), for such points we have
$$\bar{u}(\ell(s)) = u(\tau(s)\bw)\bar{u}(\bv_0) = \til u(s) \bar u({\mathbf v_0}),$$ 
and so according to Proposition
\ref{prop: hajos di}, $\ell(s)$ admits an improvement in 
Dirichlet's theorem. 

We turn to the proof of (i--iv). 
In view of Proposition \ref{prop: dense bounded}
it suffices to show that the exists a nonempty open subset of $PU$ such
that any element $g$ in the intersection of $\mathcal D$ and this open subset 
satisfies (i--iv) for some interval $I$.
\begin{comment}
Assertion (i) can be proved using the
implicit function theorem, by computing that the derivatives of the
two maps 
$$(u_0,p_0) \mapsto p_0u_0, \ \ (u_0,p_0) \mapsto u_0p_0$$ 
which are nondegenerate at
$(e, e)$, and thus define local homeomorphisms
$ U \times P \to G$ for a neighborhood of $(e,e)$. However we will give a more explicit proof as we
will require 
need more information about the solutions. 
\end{comment}

Let $p_{ij}$ denote  the
matrix entries of $p$. Then we have
\begin{align*}\label{eq: 1}
u_s p=
\left( \begin{matrix} p_{11} & p_{12} & p_{13} &0\\
p_{21}& p_{22}& p_{23}  & 0\\
p_{31}+ s p_{41} & p_{32}+sp_{42} & p_{33}+sp_{43} & sp_{44}\\
p_{41} & p_{42}& p_{43}& p_{44}
\end{matrix} \right)  .
\end{align*}
The top left $3 \times 3$ block of a product $p(s)^{-1} \til
u(s)$ is the same as that of $p(s)^{-1}$.
It follows that 
\begin{align*}
q(s)=
\left(
\begin{matrix}
a(s) & 0 \\
0 & b(s)
\end{matrix}
\right)
\quad
\mbox{with }
a(s)=b(s)\left(
\begin{matrix}
a_{11}(s) & a_{12}(s)  & a_{13}\\
a_{21}(s) & a_{22}(s) & a_{23}  \\
a_{31}(s) & a_{32}(s) & a_{33} \\
\end{matrix}
\right)
\end{align*}
where $b(s)^{-1}$ is the determinant of the top left $3\times 3$
matrix of $u_sp$, 
$a_{i1}(s), a_{i2}(s)$ are affine functions of $s$ and $a_{i3}$ are constants.
Also 
\[
\tilde u(s)=a (0, 0 , sp_{44})^{\mathrm{tr}}=sb(s)p_{44}(a_{13}, a_{23}, a_{33})^{\mathrm{tr}}.
\]
It follows that for any element  of $PU$ 
there exists an interval $I$ of $\mathbb R$ such that  (\rmnum{1}) and
(\rmnum{2}) hold.

For any $\sigma$ let $\mathfrak{u}^+_\sigma$ denote the Lie
algebra of $U^+_\sigma$ and let $\mathfrak{h}$ denote the Lie algebra
of $H$.
We claim  that the set $\mathcal S$ of elements $g\in PU$ such that 
\begin{equation}\label{eq: 4}
\text{  for any } \sigma, \, q'(0) q(0)^{-1} \notin \mathfrak{u}^+_\sigma +
\Ad(q(0)) (\mathfrak{h})
\end{equation}
is a nonempty open subset. 
Assume the claim, then there exists $g\in \mathcal D$ such that (\ref{eq: 4}) holds. 
Recall that 
$$K_1 = \bigcup_{\sigma} U^+_{\sigma} \Z^n,$$
that is a finite union of compact 6-dimensional manifolds, each of
which is a $U_\sigma^+$-orbit. Also the orbit $H\pi(g)$
is a 7-dimensional manifold, and $q(s)H\pi(g)$ is thus a closed 
$q(s)Hq(s)^{-1}$-orbit. 
If $q(0)H\pi(g)$ intersects $K_1$ at a point $x$, then
\eqref{eq: 4} implies
that the application of $q(s)$ for small nonzero $s$ maps a neighborhood of $x$ in 
$q(0)H\pi(g)$ away from 
$K_1$. Since $K_1$ is compact,
$q(s)H\pi(g)$ and $K_1$ are disjoint, and (iii) follows. 
By \eqref{eq: to be solved}, $\til u(s) \bar{u}(\mathbf {v}_0) = \til u(s)
p^{-1}\pi(g) = p(s) u_s \pi(g) $. If $t_n \to \infty$ and the sequence
$\left(f_{t_n}p(s)
u_s \pi(g) \right)_{n \geq 1}$ converges, then by \eqref{eq: prop of P0}, 
\begin{align*}
\lim_{n \to \infty} f_{t_n} p(s)  u_s \pi(g) &=  \lim_{n \to \infty}
f_{t_n} p(s)  f_{-t_n} f_{t_n} u_s\pi(g)   \\
&= \lim_{n \to \infty} q(s)
f_{t_n} u_s
x \in q(s) H\pi(g).
\end{align*}
 Thus (iv) follows from (iii). 

It remains to prove the claim. It is easy to see that the set $\mathcal S$
is open. So we only need to show that it is nonempty. 
We will show that there exists 
$g\in \mathcal S$ such that $p$ is equal to 
$$
\left(\begin{matrix} 1 & 0 & 1 & 0 \\
0 & 1 & 1 & 0 \\
0 & 0 & 1 & 0 \\
x & y & z & 1
\end{matrix}
\right),
$$
for an appropriate choice of $x, y, z$. 
Expressing $q(s)^{-1}$ using \eqref{eq: to be solved}, and taking the derivative with respect to $s$ in the equation
\[
q(s)q(s)^{-1}=e, 
\]
yields
\begin{equation}\label{eq: 5}
q'(0)q(0)^{-1}=
\left(
\begin{matrix}
x & y & z  & 0 \\
x & y & z & 0 \\
-x & -y& -z & 0\\
0 & 0 & 0 & z-x-y
\end{matrix}
\right).
\end{equation}
Computing explicitly the adjoint representation for $p$ 
we obtain:
\begin{equation}\label{eq: 6}
\Ad(q(0)) \left(\begin{matrix} a & b & 0 & 0 \\
c & d & 0 & 0 \\
0 & 0 & e & f \\
0 & 0 & g & h \end{matrix} \right) =
\left(\begin{matrix} a & b & -a-b+e & f \\
c & d & -c-d+ e & f \\
0 & 0 & e & f \\
0 & 0 & g & h
\end{matrix} \right).
\end{equation}
That is, an element of $\Ad(q(0))(\mathfrak{h})$ can be written as the
right hand side of \eqref{eq: 5}, for an appropriate choice of $a, b,
c, d, e, f, g, h$ (with $a + d + e + h=0$). 

We will show that for each $\sigma$, the failure of 
\eqref{eq: 4} leads to a nontrivial linear relation among the
$x,y,z$. So taking $x,y,z$ which do not solve these finitely many
linear relations forces \eqref{eq: 4}. For instance, if
$E_{31} \notin \mathfrak{u}^+_\sigma$, then examining the (31) entry in \eqref{eq: 5} and
\eqref{eq: 6} leads to $x=0$. Similarly $E_{32} \notin
\mathfrak{u}^+_\sigma$ leads to $y=0$. For a more interesting case consider the
case when both $E_{12}, E_{13}$ do not belong to $\mathfrak{u}^+_\sigma$. From two of the diagonal entries
in \eqref{eq: 5}, \eqref{eq: 6} 
we obtain $a=x, e = -z$. From the (12) entry we obtain
$b=y$, and 
from the (13) entry we find $-a-b+e = z$. We have four linear equations for the three variables $a, b, e$,
and they only have a solution when $0=x+y+2z$. This is the sought-for
linear relation. 

By similar arguments one deals with the case when both $E_{21},
E_{23}$ are not in $\mathfrak{u}^+_\sigma$, and since for each $\sigma$, one of the two 
elements
$E_{12}, E_{21}
$ 
is contained in $\mathfrak{u}^+_\sigma$, these cases cover all
possibilities. This concludes the proof. 
\end{proof}

\begin{comment}
\subsection{Example for Theorem \ref{thm;equiline}}
In this subsection we show that hypothesis \eqref{eq: weird condition}
is necessary for the validity of Theorem \ref{thm;equiline}. In fact
we show that 
it is not enough to replace \eqref{eq: weird condition} with the
requirement that there is a compact $X \subset X$ and $s_0$ for which
there are $t_n \to
\infty$ such that $f_{t_n} \bar{u}(s_0, as_0+b) \in K$. 
\end{comment}

\end{document}